\documentclass[12pt]{article}
\usepackage[margin=2.5cm]{geometry}
\usepackage{amsmath,amsthm,amsfonts,amssymb}
\usepackage{amsfonts,epsf,amsmath,tikz}
\usepackage{graphicx,latexsym}
\usepackage{color}
\usepackage{float}
\usepackage[ruled,vlined,linesnumbered]{algorithm2e}

\newtheorem{thm}{Theorem}[section]

\newtheorem{cor}[thm]{Corollary}
\newtheorem{lem}[thm]{Lemma}
\newtheorem{prop}[thm]{Proposition}

\newtheorem{defn}[thm]{Definition}

\newtheorem{obs}{Observation}

\newtheorem{claim}[thm]{Claim}

\newcommand{\gt}{{\rm gt}}

\newcommand{\cc}{{\rm cor}}
\newcommand{\gpack}{{\rm gpack}}

\newcommand{\SM}{{\rm SM}}

\newcommand{\cp}{\,\square\,}

\begin{document}

	\title{The geodesic transversal problem on some networks}
	\date{}

	\author{
	Paul Manuel$^{a}$
	\and
	Bo\v{s}tjan Bre\v{s}ar$^{b,c}$
	\and
	Sandi Klav\v zar$^{b,c,d}$
}

	\maketitle

\begin{center}
	$^a$ Department of Information Science, College of Life Sciences, Kuwait University, Kuwait \\
	{\tt pauldmanuel@gmail.com}\\
	\medskip
	
	$^b$ Faculty of Natural Sciences and Mathematics, University of Maribor, Slovenia\\
	{\tt bostjan.bresar@um.si}\\
	\medskip

	$^c$ Institute of Mathematics, Physics and Mechanics, Ljubljana, Slovenia\\
	\medskip

	$^d$ Faculty of Mathematics and Physics, University of Ljubljana, Slovenia\\
	{\tt sandi.klavzar@fmf.uni-lj.si}\\
\end{center}

\begin{abstract}
A set $S$ of vertices of a graph $G$ is a geodesic transversal of $G$ if every maximal geodesic of $G$ contains at least one vertex of $S$. We determine a smallest geodesic transversal in certain interconnection networks such as meshes of trees, and some well-known chemical structures such as silicate networks and carbon nanosheets. Some useful general bounds for the corresponding graph invariant are obtained along the way.
\end{abstract}	


\smallskip


\section{Introduction}
One of the fundamental concepts in graph theory is that of a shortest path between two vertices, also called  a {\em geodesic}. A geodesic is \textit{maximal} if it is not a subpath of another geodesic. 
A set $S$ of vertices of $G$ is a \textit{geodesic transversal} of $G$ if every maximal geodesic of $G$ contains at least one vertex of $S$. We also say that a vertex $x\in S$ {\em hits} geodesic $P$ if $x$ is a vertex of $P$. 
The \textit{geodesic transversal number} of $G$, denoted by $\gt(G)$, is the minimum cardinality of a geodesic transversal of $G$. A set $S$ of vertices is a {\em $\gt$-set} of $G$ if $S$ is a minimum cardinality geodesic transversal of $G$. The \textit{geodesic transversal problem} of $G$ is to find a $\gt$-set of $G$.

Interestingly, this concept was introduced recently by two independent research groups~\cite{MaBr21,PeSe21a,PeSe21b}. The geodesic transversal problem was shown to be NP-complete~\cite{MaBr21,PeSe21a} and polynomially solvable for trees~\cite{MaBr21,PeSe21a} and some cactus families~\cite{MaBr21}. Peterin and Semani\v{s}in~\cite{PeSe21b} have shown some interesting results on the geodesic transversal problem of the join and the lexicographic product of graphs. A motivation for studying the geodesic transversal problem was illustrated in~\cite{MaBr21,PeSe21a}, where applications in (large-scale) network theory and complex systems control were indicated. As an additional motivation, we present another connection of our study to a problem that arises in networks called the shortest path union cover problem as introduced by Boothe et al.~\cite{BoDv07}. If $G$ is a graph, then a set $S$ of vertices is called a {\em shortest path union cover} if the shortest paths that start at the vertices of $S$ cover all the edges of $G$. The \textit{shortest path union cover problem} is to find a shortest path union cover of minimum cardinality. Since all geodesics that start at the vertices of a geodesic transversal of $G$ cover all the edges of $G$, every geodesic transversal is a shortest path union cover. However, the converse is not true. Any leaf of a tree is a shortest path union cover and hence the shortest path union cover number of a tree is just $1$. On the other hand, the computation of the geodesic transversal number of a tree is not straightforward~\cite{MaBr21,PeSe21a}. 

The geodesic transversal problem is also of applicable nature, which is another motivation for our study. Notably, shortest paths are the basis of distance-based topological indices that are extensively studied and applied in chemical graph theory, let us mention just the notorious Wiener Index.  (The reader can start exploring this vast field by reading, for example, the following articles~\cite{cavaleri-2022, chen-2022, spiro-2022}.) Since, according to Bellman's principle, each sub-path of a shortest path will itself be a shortest path, maximal geodesics have a special distinguished role. To our knowledge, this has not been investigated in mathematical chemistry, but in our opinion the role of maximal geodesics would certainly be worth investigating there.

In this paper, we first present two simple, yet useful bounds, an upper and a lower bound for the geodesic transversal number of an arbitrary graph, and demonstrate their sharpness. They are given in the following section and are used in several proofs of later sections. In Section~\ref{sec:mesh}, we study meshes of trees, which are constructed by combining square grids with complete binary trees, and we determine their geodesic transversal number. In Section~\ref{sec:geo-pack-leaf-attach}, we consider $\gt$-sets with respect to two operations of attaching leaves to vertices of a graph. One of the operations is used in Section~\ref{sec:silicate}, where we solve the geodesic transversal problem for the well-known chemical structure of silicate networks. 

\section{Basic bounds}

The following upper bound for the geodesic transversal number follows from the fact that each (maximal) geodesic has two end-vertices and to hit all maximal geodesics, it suffices to hit all but one end-vertex of maximal geodesics of a graph.

\begin{obs}
\label{ob:endvertices}
If $G$ is a graph, and $X$ is the set of end-vertices of all maximal geodesics of $G$, then $\gt(G)\le |X|-1$. In addition, if there is a subset $X'\subset X$ such that no maximal geodesic lies between two vertices from $X'$, then $\gt(G)\le |X|-|X'|$. 
\end{obs}

The second statement of Observation~\ref{ob:endvertices} follows from the fact that every maximal geodesic has an end-vertex in $X-X'$, thus vertices of $X-X'$ hit all maximal geodesics of $G$. 
The (first) bound in Observation~\ref{ob:endvertices} is sharp as can be seen by paths $P_n$, where $\gt(P_n)=1$, and complete graphs $K_n$, where $\gt(K_n)=n-1$. On the other hand, the bound need not be useful, as one can see in stars, since $\gt(K_{1,n})=1$, while the bound is $n-1$. 

To demonstrate the proof techniques by using also the bound in Observation~\ref{ob:endvertices}, we next present a result concerning a structure that arises in mathematical chemistry~\cite{XuLi17}. Fig.~\ref{fig:FTriCabNano} shows the $3$-layer and the $4$-layer triangular carbon nanosheet~\cite{HoMa19, XuLi17}, and the definition of the $r$-layer triangular carbon nanosheet should be clear from these two cases. Next, we establish the exact value of the geodesic transversal number of the $r$-layer triangular carbon nanosheet. Recall that a {\em leaf} or a {\em pendant vertex} in a graph $G$ is a vertex having one neighbor in $G$,  while a {\em support vertex} is a vertex adjacent to a leaf.
\begin{figure}[ht!]
	\begin{center}
		\scalebox{0.59}{\includegraphics{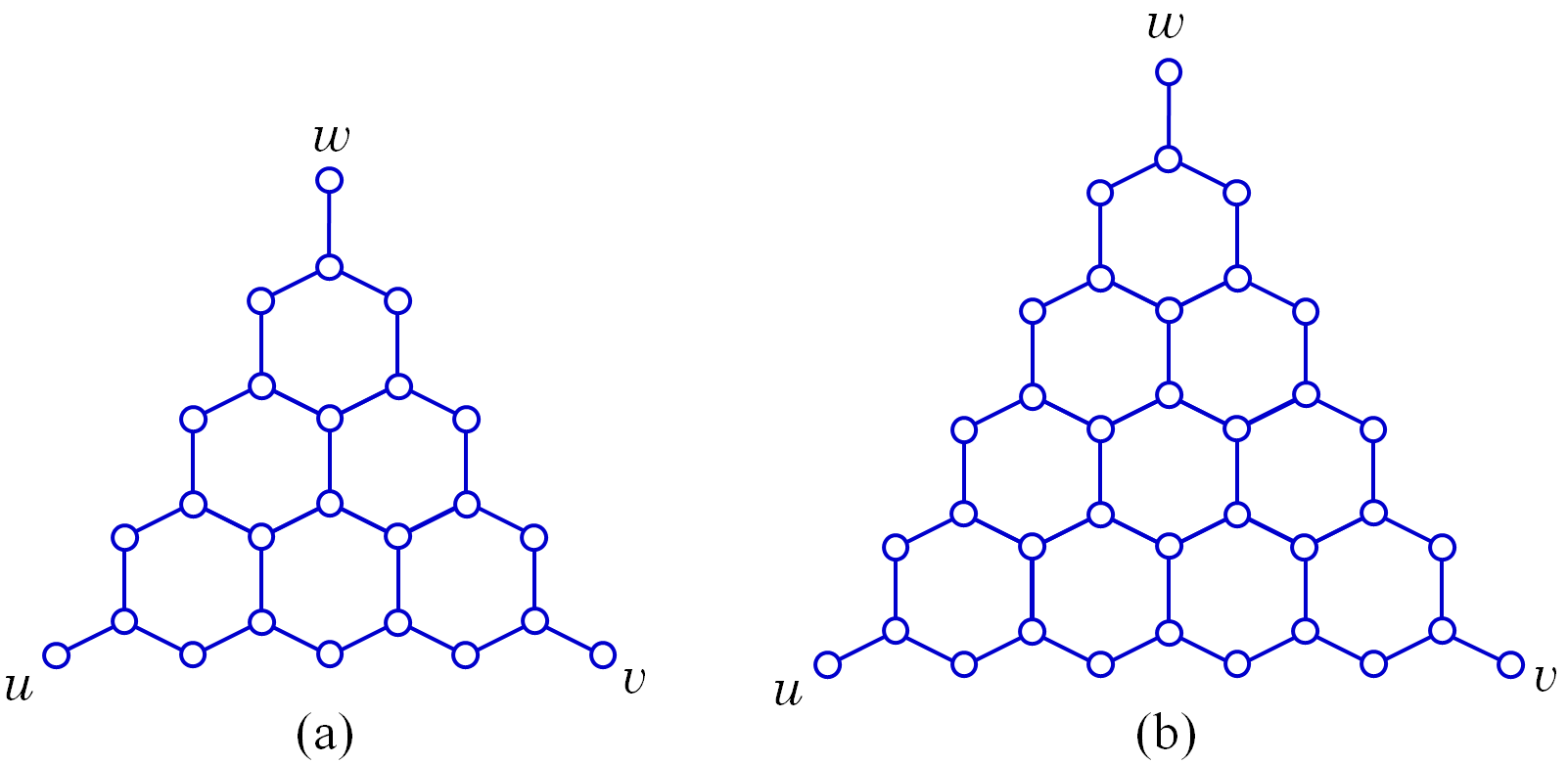}}
	\end{center}
	\caption{(a) A 3-layer triangular carbon nanosheet, \quad (b) A 4-layer triangular carbon nanosheet.}
	\label{fig:FTriCabNano}
\end{figure}

\begin{prop}
	\label{PgtTriCarNano}
	If $G$ is an $r$-layer triangular carbon nanosheet, then $\gt(G)=3$.
\end{prop}
\begin{proof}
	We denote by $u$, $v$ and $w$ the three corner vertices of $G$; these are the only pendant vertices of $G$, and let $Y=\{u,v,w\}$. Let $X'$ be the set of vertices of degree $2$.  Note that the set of end-vertices of maximal geodesics of $G$ is $X=Y\cup X'$. Indeed, if $x$ is an internal vertex of degree $3$, then $x$ cannot be an end vertex of a maximal geodesic, since such a geodesic can be extended through one of the three edges incident with $x$. Similarly, if $x$ is an external vertex of degree $3$ which lies on a geodesic $P$, then one of its two neighbors of degree $2$ can be added to $P$ so that the resulting path is still a geodesic. From these facts we readily infer that there are two types of maximal geodesics in $G$: (1) connecting two of the corner vertices, and (2) connecting a corner vertex with a vertex of degree $2$ that lies in the (opposite) geodesic between the other two corner vertices. (In addition, all the maximal geodesics are of diametral length.) 
		
	For the upper bound, $\gt(G)\le 3$, we use the second statement of Observation~\ref{ob:endvertices}. Indeed, by the previous paragraph, there are no maximal geodesics between two vertices in $X'$, hence $\gt(G)\le |X|-|X'|=|Y|=3$. In particular, the vertices of $Y$ form a geodesic transversal. By a simple case analysis (taking into consideration all possibilities of positioning two vertices in a potential geodesic transversal), one finds that $\gt(G)>2$.
\end{proof}

For the lower bound on the geodesic transversal number we introduce the following concept. A {\em geodesic packing} of $G$ is a set of maximal geodesics $\{P_1,\ldots, P_k\}$, which are mutually vertex-disjoint. The {\em geodesic packing number} of $G$ is the maximum cardinality of a geodesic packing of $G$ and is denoted by $\gpack(G)$. Clearly, every geodesic transversal $S$ of $G$ must hit each geodesic in a (maximum) geodesic packing $P$ in $G$, and pairwise distinct vertices from $S$ are needed to hit pairwise distinct geodesics from $P$. We infer the following result. 

\begin{obs}
	\label{obLowBound}
	Given a graph $G$, $\gt(G)\geq \gpack(G)$. 
\end{obs}

The bound in Observation~\ref{obLowBound} is again sharp in paths, where $\gt(P_n)=\gpack(P_n)=1$. It can fail to be useful, like in complete graphs $K_{n}$, where $\gpack(K_n)=\lfloor\frac{n}{2}\rfloor$, while $\gt(K_n)=n-1$. A large class of graphs $G$ in which $\gt(G)=\gpack(G)$ will be presented in Section~\ref{sec:geo-pack-leaf-attach}.

\section{Mesh of trees}
\label{sec:mesh}

In this section, we establish the geodesic transversal number of mesh of trees, which are constructed by combining 2-dimensional square grids with complete binary trees. For this purpose, we first consider the two classes that present building blocks of a mesh of trees. 

We start by presenting the class of square grids. First, the {\em Cartesian product} of graphs $G$ and $H$ is the graph with $V(G)\times V(H)$ as its vertex set, while $(g,h)(g',h')\in E(G\cp H)$ if and only if ($gg'\in E(G)$ and $h=h'$) or ($g=g'$ and $hh'\in E(H)$. See the monograph~\cite{HIK-11} for more on graph products. Now, a {\em square grid} is a graph $P_r\cp P_s$ for some $r,s\in \mathbb{N}$. In many situations, a graph may have an exponential number of maximal geodesics. In particular, this holds for square grids $P_r \cp P_r$. However, all maximal geodesics in $P_r \cp P_r$ can be hit by only two vertices. 

A {\em subdivided star} is the graph obtained from $K_{1,n}$, $n\ge 1$, by subdividing each of the edges of $K_{1,n}$ an arbitrary number of times (possibly zero, and different edges may be subdivided different number of times). As proved in~\cite[Proposition 3.2]{MaBr21}, $\gt(G)=1$ if and only if $G$ is a subdivided star.  

\begin{prop}
	\label{Tgt2DGrid}
	The geodesic transversal number of a $2$-dimensional square grid $P_r \cp P_s$ is $2$.
\end{prop} 
\begin{proof}
Recall a well known fact that in any Cartesian product a path $P$ is a geodesic if and only if the projections of $P$ are geodesics in each of the factors (where repetitions of vertices in the projected paths are disregarded). We readily infer that only the corner vertices (that is, the vertices of degree $2$) are end-vertices of maximal geodesics in $P_r \cp P_s$. (By Observation~\ref{ob:endvertices}, this implies $\gt(P_r \cp P_s)\le 3$.) Moreover, the geodesic between any two corner vertices in the same row (respectively, same column) is not maximal, since it can be extended to a diametral path. Hence by using the second statement of Observation~\ref{ob:endvertices} we derive $\gt(P_r \cp P_s)\le 2$, and the corner vertices in the same row (or column, respectively) form a geodesic transversal. Since the grid is not a subdivided star, we infer $\gt(P_r \cp P_s)\ge 2$ by~\cite[Proposition 3.2]{MaBr21}.
\end{proof}

\begin{figure}[ht!]
	\begin{center}
		\scalebox{0.5}{\includegraphics{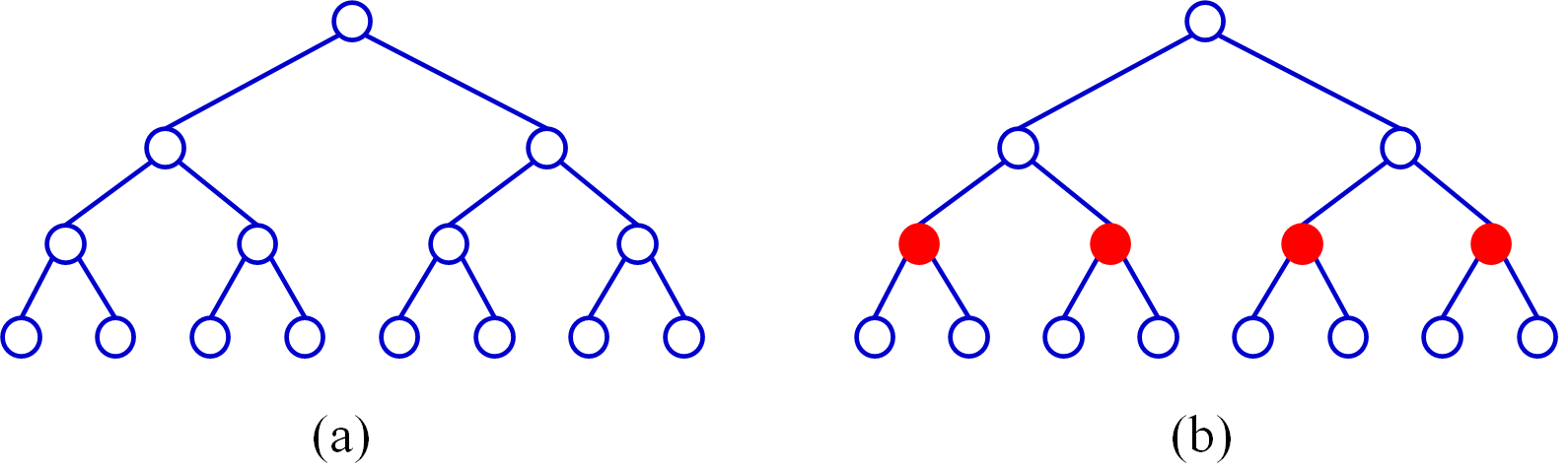}}
	\end{center}
	\caption{(a) A complete binary tree \quad \quad \quad (b) The red bullets form a geodesic transversal}
	\label{fig:FCompBinTrees}
\end{figure}

The {\em complete binary tree} of {\em dimension $r$} is the rooted tree in which all ($2^r$) leaves have the same depth $r$, and all interior vertices have $2$ children. (Fig.~\ref{fig:FCompBinTrees}(a) presents the complete binary tree of dimension $3$.)

\begin{prop}
	\label{TgtCBT}
If $G$ is an $r$-dimensional complete binary tree, then $\gt(G) = 2^{r-1}$.
\end{prop}
\begin{proof}
Consider the set of the maximal geodesics of length $2$ in $G$ which connect two leaves through their common support vertex.
There are $2^{r-1}$ such maximal geodesics, and they are pairwise disjoint, thus $\gpack(G)\ge 2^{r-1}$. By Observation~\ref{obLowBound}, $\gt(G) \geq 2^{r-1}$.   
Note that the end-vertices of any maximal geodesic in $G$ are leaves, since (maximal) geodesics in a tree are all its (maximal) paths. Hence we infer that the set $S$ of support vertices of $G$ (see Fig.~\ref{fig:FCompBinTrees}(b)) is a geodesic transversal. Thus, $\gt(G) \leq |S| = 2^{r-1}$.
\end{proof}
\begin{figure}[ht!]
	\begin{center}
		\scalebox{0.5}{\includegraphics{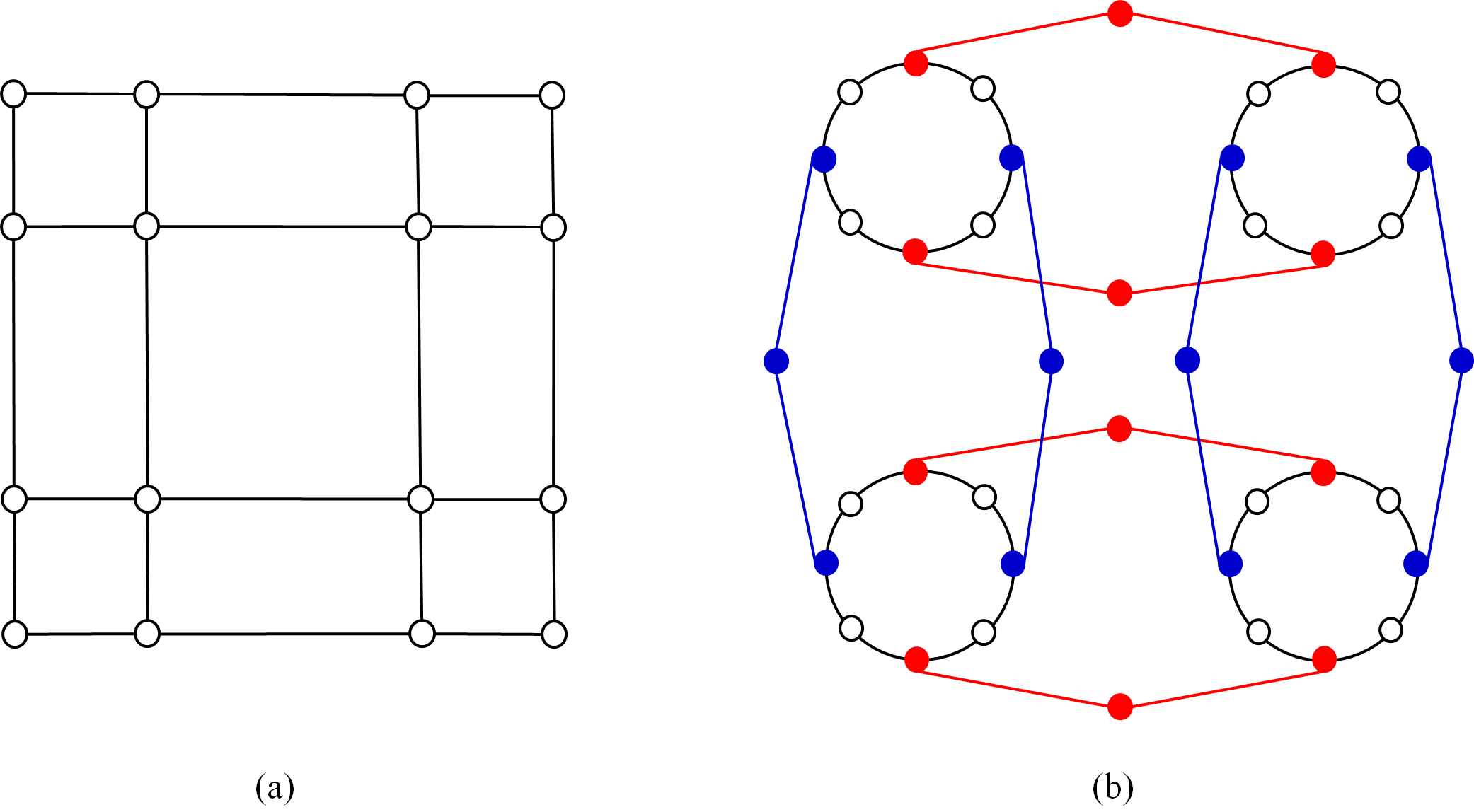}}
	\end{center}
	\caption{(a) $2^2\times 2^2$ square grid. 
		(b) An $N\times N$ mesh of trees where $N=2^2$. 
				The grid vertices are white, the row tree vertices red, and the column tree vertices are in blue.}
	\label{MeshOfTrees4x4}
\end{figure}

A mesh of trees is a well-known interconnection parallel architecture~\cite{Leig91,Xu02}. The advantage of the mesh of trees is that it combines the features of both $2$-dimensional square grids as well as complete binary trees. The mesh of trees was originally introduced by Leighton~\cite{Leig91}. It is now formally defined here. Let $MoT(r)$ denote the $N\times N$ mesh of trees where $N= 2^r$ and $r$ is a positive integer. The graph $MoT(r)$ is constructed on top of the $2$-dimensional square grid, where each row and column are replaced by a complete binary tree (and the original edges of the grid are deleted). The vertices of the $N\times N$ grid are marked in white and play the role of leaves of the corresponding complete binary trees both in rows and columns. The row trees are marked in red and the column trees are marked in blue (with the exception of their leaves, which are colored white, as mentioned earlier). See Fig.~\ref{MeshOfTrees4x4}. Let us now identify some geodesic transversal of $MoT(r)$. 
\begin{lem}
	\label{LMoT}
	Let $MoT(r)$ be the $N\times N$ mesh of trees, where $N =2^r$. For each maximal geodesic $P$ of $MoT(r)$ there exists an $8$-cycle $C$ such that $P$ contains at least one blue vertex and one red vertex of $C$. 
\end{lem}
\begin{proof}
Similarly as in the proof of Proposition~\ref{TgtCBT} we note that no internal vertex of a complete binary tree (either blue or red) is an end-vertex of a maximal geodesic of $MoT(r)$. Therefore, end-vertices of any maximal geodesic of $MoT(r)$ must lie in $8$-cycles of $MoT(r)$.  Given a maximal geodesic $P$, we distinguish the following three cases. 
 
\begin{itemize}
	\item[] Case I - $P$ lies entirely inside an $8$-cycle $C$. Since $P$ is maximal, it must be of length 4. Thus, $P$ contains at least one blue vertex and one red vertex of $C$.  
		\item[] Case II - $P$ intersects exactly two $8$-cycles $C_1$ and $C_2$. In other words, $P$ starts in $C_1$, passes through vertices of a complete binary tree and ends in $C_2$. It is clear that both cycles $C_1$ and $C_2$ are in a single row or in a single column (see Fig.~\ref{MeshOfTreesTypeII}, where (a) and (b) represent the cases, where two cycles are in the same row and column, respectively). Without loss of generality, suppose that $P$ connects the cycles $C_1$ and $C_2$ via vertices of a red complete binary tree (case (a) in the figure). If no blue vertex in $C_1$ is contained in $P$, then at most two vertices in $C_1$ are contained in $P$ (a red vertex and possibly a white vertex). However, since $P$ is a maximal geodesic, this implies that at least four vertices from $C_2$ are contained in $P$, one of which is blue.  In either case, a blue and a red vertex of some $8$-cycle are contained in $P$. (For the case (b) in the figure one uses symmetric arguments.)	
	\begin{figure}[ht!]
	\begin{center}
		\scalebox{0.43}{\includegraphics{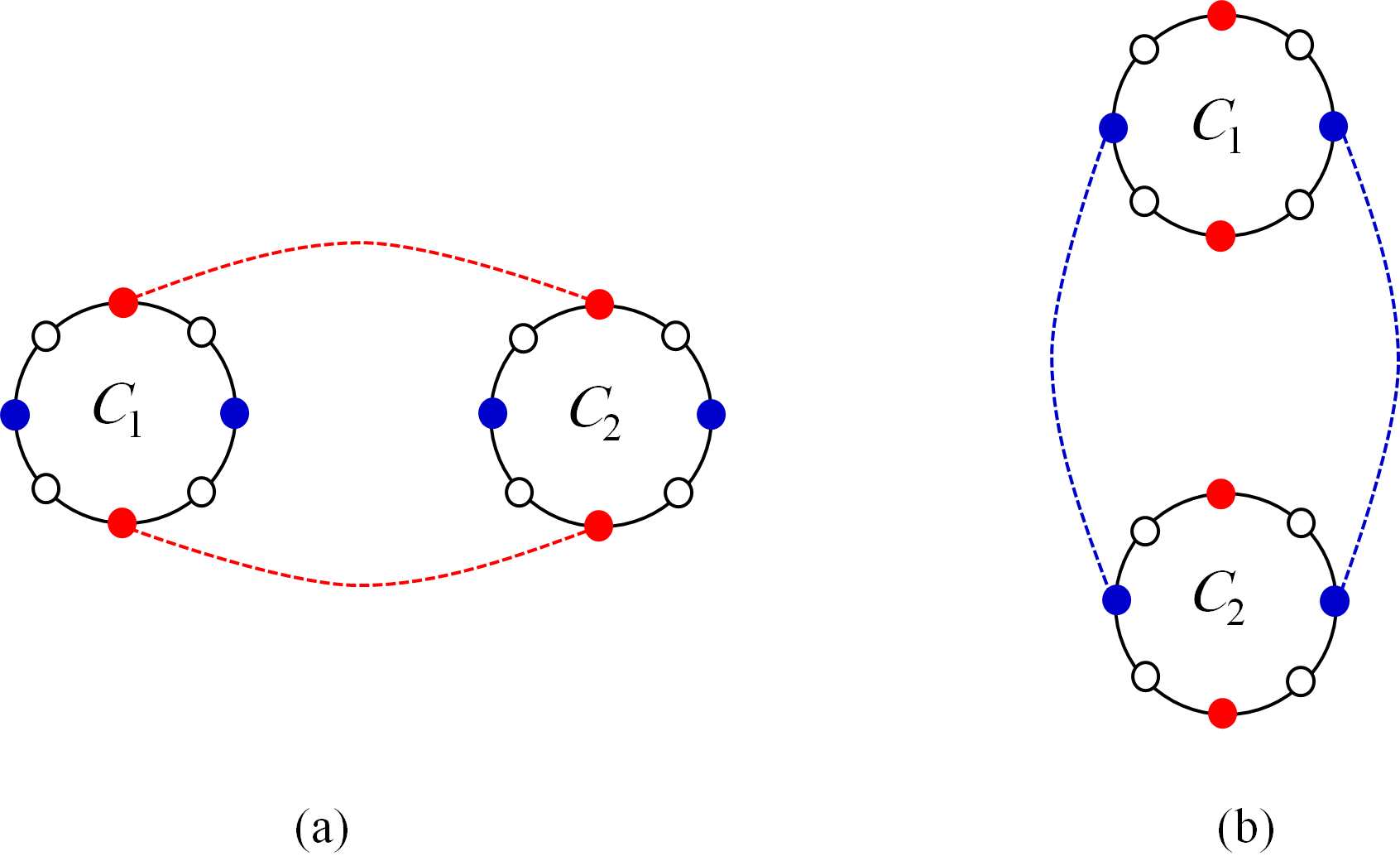}}
	\end{center}
	\caption{Case II of Lemma \ref{LMoT}.}
	\label{MeshOfTreesTypeII}
\end{figure}
	\item[] Case III - $P$  intersects more than two $8$-cycles. In this case, the cycles $C_1$ and $C_2$, which contain end-vertices of $P$, are not in the same column nor in the same row. Indeed, if $C_1$ and $C_2$ are in the same row (respectively column), then all maximal geodesics that go from a vertex in $C_1$ to a vertex in $C_2$ pass through red (respectively, blue) vertices of the corresponding complete binary tree that connects $C_1$ with $C_2$. Consequently, we are in Case II, where only two $8$-cycles are involved in $P$.
	Therefore, there exists an $8$-cycle $C$, different from $C_1$ and $C_2$, such that $P$  enters into $C$ through a red (respectively, blue) vertex of $C$ and exits $C$ in a blue (respectively, red) vertex of $C$. 
\end{itemize}  
This concludes the proof.
\end{proof}
\begin{thm}
\label{TMoT}
	If $MoT(r)$ is an $N\times N$ mesh of trees, where $N =2^r$, then $\gt(G) = 2^{2r-1}$.
\end{thm}
\begin{proof}
 As noted in the proof of Lemma~\ref{LMoT}, end-vertices of any maximal geodesic in the graph lie in $8$-cycles. There are $2^{r-1}\cdot 2^{r-1}$ such vertex-disjoint $8$-cycles in $G$. 
 Each $8$-cycle $C$ contains four maximal geodesics between white vertices, two connecting each pair of antipodal white vertices.  Clearly, to hit these four maximal geodesics of $C$, one needs at least two vertices in $C$.  Since the $8$-cycles are vertex-disjoint, we infer $\gt(G)\geq 2\cdot(2^{r-1}\cdot 2^{r-1})$ = $2^{2r-1}$.
		
Let $R$ denote the set of all red vertices of all $8$-cycles in $MoT(r)$. By Lemma \ref{LMoT}, $R$ is a geodesic transversal of $MoT(r)$. There are $2^r$ row (red) complete binary trees in $MoT(r)$, and each of the binary trees contains $2^{r-1}$ red vertices in $8$-cycles. Hence, $|R|=2^{r}\cdot 2^{r-1}=2^{2r-1}$, and we infer $\gt(G) \leq 2^{2r-1}$.
\end{proof}

\section{Geodesic transversal and leaf attachment operations}
\label{sec:geo-pack-leaf-attach}

In this section, we consider two operations of attaching leaves to vertices of a graph, which behave nicely with respect to geodesic packing. In particular, the second operation will be applied in the next section for determining the geodesic transversal number of silicate networks.

The first operation, known as {\em corona} of the graph, is defined as follows.
Given a graph $G$, let $\cc(G)$ be the graph with $V(\cc(G))=V(G)\cup\{v':\, v\in V(G)\}$ and $E(\cc(G))=E(G)\cup\{vv':\, v\in V(G)\}$. In other words, to each vertex $v$ of $G$ we attach a leaf $v'$ and get the corona, $\cc(G)$, of $G$. 
Note that in the literature one can also find the notation $G\circ K_1$ for the corona of $G$. 

Using the notation of West~\cite{west-2001}, $\alpha'(G)$ stands for the {\em matching number} of a graph $G$, that is, the size of a largest matching in $G$. Also,  $\beta(G)$ is the standard notation for the {\em vertex cover number} of $G$, which is the smallest set of vertices that hit all edges of $G$.  

\begin{prop}
	\label{prp:corona}
	If $G$ is a graph, then $\gpack(\cc(G))=\alpha'(G)$ and $\gt(\cc(G))=\beta(G)$.
\end{prop}
\proof
Let $M$ be a maximum matching in $G$. Then the collection of paths $${\cal P}=\{(x',x,y,y'):\, xy\in M, \textrm{ and } x' \textrm{ (}y') \textrm{ is the leaf neighbor of } x \textrm{ (} y, \textrm{ respectively)}\}$$  
is a geodesic packing of $\cc(G)$. This implies $\gpack(\cc(G))\ge \alpha'(G)$. For the reverse inequality, note that the end-vertices of every maximal geodesic in $\cc(G)$ are leaves that are attached to distinct vertices of $G$. Now, if $P$ and $Q$ are two disjoint maximal geodesics in $\cc(G)$, then they pass two distinct edges of $G$, which have no end-vertex in common. Hence, the size of a geodesic packing of $\cc(G)$ is not larger than the size of a maximum matching in $G$, which proves the first equality of the proposition.

For the second equality note that the set of vertices forming a vertex cover in $G$ is a geodesic transversal in $\cc(G)$. Hence, $\gt(\cc(G))\le \beta(G)$. If a set $S\subset V(\cc(G))$ is smaller than $\beta(G)$, then it is easy to see that there exists an edge $xy\in E(G)$, such that $\{x,y,x',y'\}\cap S=\emptyset$, where $x'$ ($y'$) is the leaf attached to $x$ ($y$, respectively). Since $(x',x,y,y')$ is a maximal geodesic in $\cc(G)$, we infer that $S$ is not a geodesic transversal in $\cc(G)$. Thus, $\gt(\cc(G))\ge \beta(G)$, and the proof is complete. \qed

\medskip

By the famous K\"{o}nig-Egervary theorem, which states that $\alpha'(G)=\beta(G)$ in any bipartite graph $G$, we get the following consequence of Proposition~\ref{prp:corona}. 

\begin{cor}\label{c:Konig}
	If $G$ is a bipartite graph, then $\gpack(\cc(G))=\gt(\cc(G))$. 
\end{cor}

\medskip

The second operation of this section is defined as follows. Given a graph $G$, let $M$ be a (not necessarily maximum) matching in $G$, $V(M)$ the set of  end-vertices of edges from $M$, and let $N$ = $V(G)-V(M)$. See Fig.~\ref{fig:FExample1}.
Let $G_M$ be the graph obtained from $G$ by attaching a single leaf to each vertex in $V(M)$, and attaching two leaves to each vertex of $N$. In other words, $V(G_M)$ = $V(G) \cup \{x':\, x\in V(M)\}\cup \{u',u'':\, u\in N\}$, while $E(G_M) = E(G) \cup$ $\{xx':\, x\in V(M)\}\cup \{uu', uu'':\, u \in N\}$.

\begin{figure}[ht!]
	\begin{center}
		\scalebox{0.30}{\includegraphics{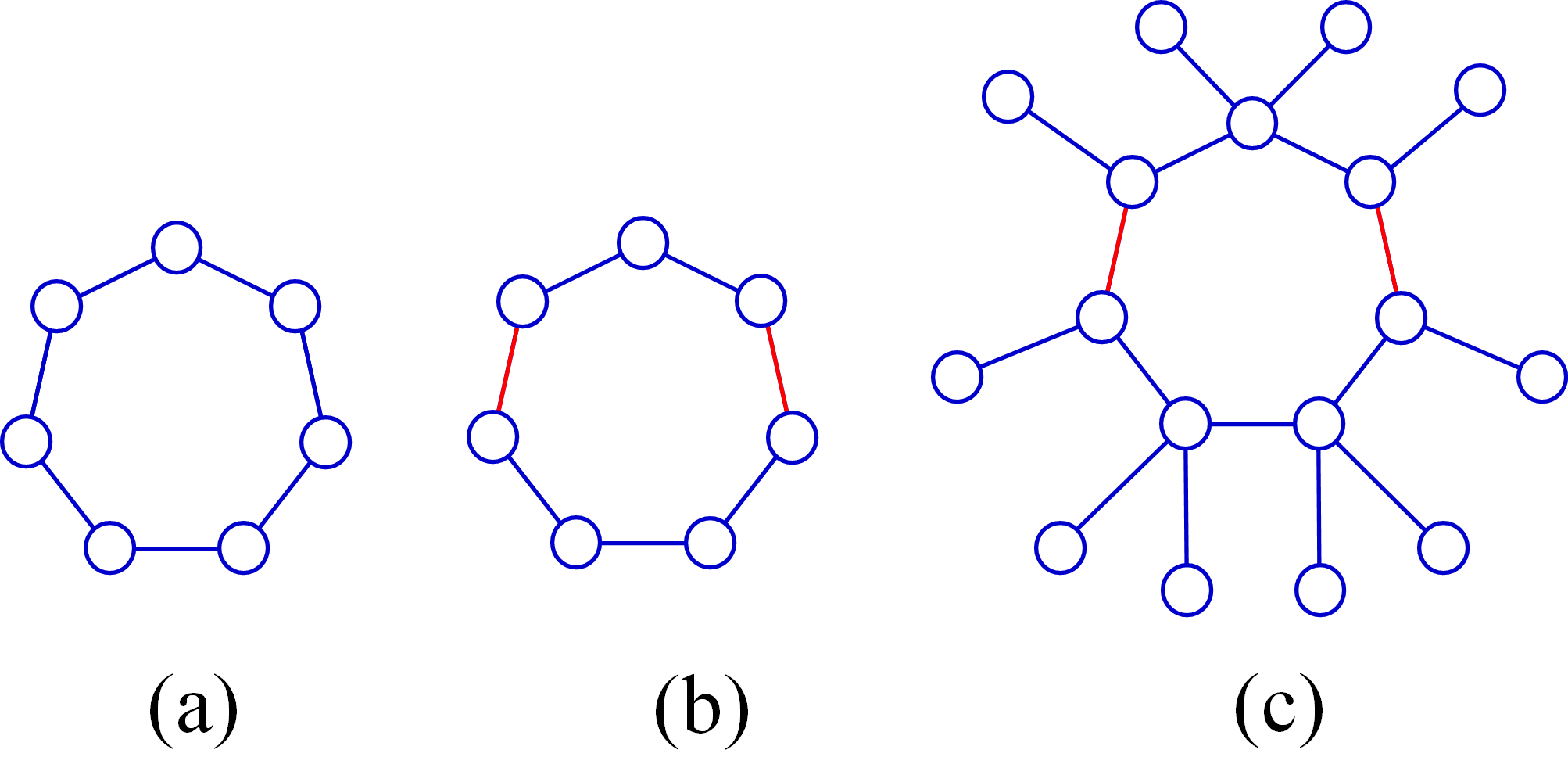}}
	\end{center}
	\caption{(a) An example graph $G$. (b) The red edges form a matching $M$.  (c) Graph $G_M$: to each vertex incident with an edge of $M$ a leaf is attached, to other vertices two leaves are attached.}
	\label{fig:FExample1}
\end{figure}

\begin{prop}
	\label{prp:gpackGM}
	If $G$ is a graph, and $M$ a matching in $G$, then $\gpack(G_M)=|V(G)|-|M|$. If, in addition, there exists $M'\subset V(M)$ such that each edge in $M$ has exactly one end-vertex in $M'$, and $M'$ is independent in $G$, then $\gt(G_M)=|V(G)|-|M|$.
\end{prop}
\proof
First, we present maximal geodesics that partition $V(G_M)$. There are two types of these maximal geodesics. First, for any $xy\in M$, let $P_{xy}=(x',x,y,y')$ be the path from the leaf neighbor of $x$ to the leaf neighbor of $y$. Second, for any $u\in N$, let $P_u=(u',u,u'')$ be the path from one leaf neighbor of $u$ to the other leaf neighbor of $u$. Clearly, $P_{xy}$ and $P_u$ are maximal geodesics for every $xy\in M$ and every $u\in N$. In addition, ${\cal P}=\{P_{xy}:\, xy\in M\}\cup \{P_u:\,u\in N\}$ is a collection of vertex disjoint maximal geodesics such that each vertex of $G_M$ belongs to exactly one of these geodesics. Clearly, $|{\cal P}|=|M|+|N|$, which implies $\gpack(G_M)\ge |M|+|N|$. 
On the other hand, note that any geodesic between any pair of distinct leaves is a maximal geodesic of $G_M$. Since there are $2|M|+2|N|$ leaves in $G_M$, we get $\gpack(G_M)\le |M|+|N|$. Combining both inequalities, we infer $$\gpack(G_M)=|M|+|N|.$$ Since $2|M|+|N|=|V(G)|$, we get $\gpack(G_M)=|V(G)|-|M|$. 

By Observation~\ref{obLowBound}, we immediately get $\gt(G_M)\ge |M|+|N|$. To prove the second statement in the proposition we present a geodesic transversal of $G_M$ with $|M|+|N|$ vertices. By the assumption of the statement, there exists $M'\subset V(M)$ such that each edge in $M$ has exactly one end-vertex in $M'$, and $M'$ is independent in $G$. Let $M''=V(M)\setminus M'$, and we claim that $N\cup M''$ gives the desired result. 
\begin{claim}
	\label{cla:claim1}
	$N\cup M''$ is a geodesic transversal of $G_M$.
\end{claim}
\proof (of Claim) Let $P$ be a maximal geodesic in $G_M$, and let $v$ be one of its end-vertices. As noted earlier in this proof, $v$ is a leaf. If $v$ is adjacent to a vertex in $M''\cup N$, then we are done. Otherwise, the neighbor $v_1$ of $v$ is an end-vertex of an edge in $M$, and $v_1\in M'$. Since $v_1$ is adjacent to only one leaf, we infer that the neighbor $v_2$ of $v_1$ that lies on $P$ is not a leaf. Since $M'$ is independent, we infer that $v_2\in M''\cup N$. Hence $P$ is hit by $v_2$, which proves the claim. $(\blacksquare)$

\medskip

By Claim \ref{cla:claim1}, we infer $\gt(G_M)\le |M|+|N|$, implying $\gt(G_M)=|M|+|N|=|V(G)|-|M|$. \qed

\section{Hexagonal silicate sheets}
\label{sec:silicate}
An $SiO_4$ tetrahedron consists of one silicon atom and 4 oxygen atoms, where the silicon atom is bonded to the 4 oxygen atoms. See Fig.~\ref{fig:FHexSil1}. A silicate sheet is a $2$-dimensional array of $SiO_4$ tetrahedra~\cite{CiBa11}. Two tetrahedra of a silicate sheet are linked together through a shared oxygen atom. In a silicate sheet, every silicon atom has exactly 4 oxygen bonds and every oxygen atom has at most 2 silicon bonds. A silicate sheet is formed by linking each tetrahedron to at most three other tetrahedra. In Fig.~\ref{fig:FHexSil1}, the oxygen nodes are in red color and the silicon nodes are in light blue color.  Chemists~\cite{KaKa06, TeSa16} do not consider the “bond between two oxygen nodes” as an edge because the bond is weak. An edge of a silicate sheet is the bond between a silicon node and an oxygen node. The bonds between two oxygen nodes are ignored and are not considered as edges. 

\begin{figure}[ht!]
	\begin{center}
		\scalebox{0.43}{\includegraphics{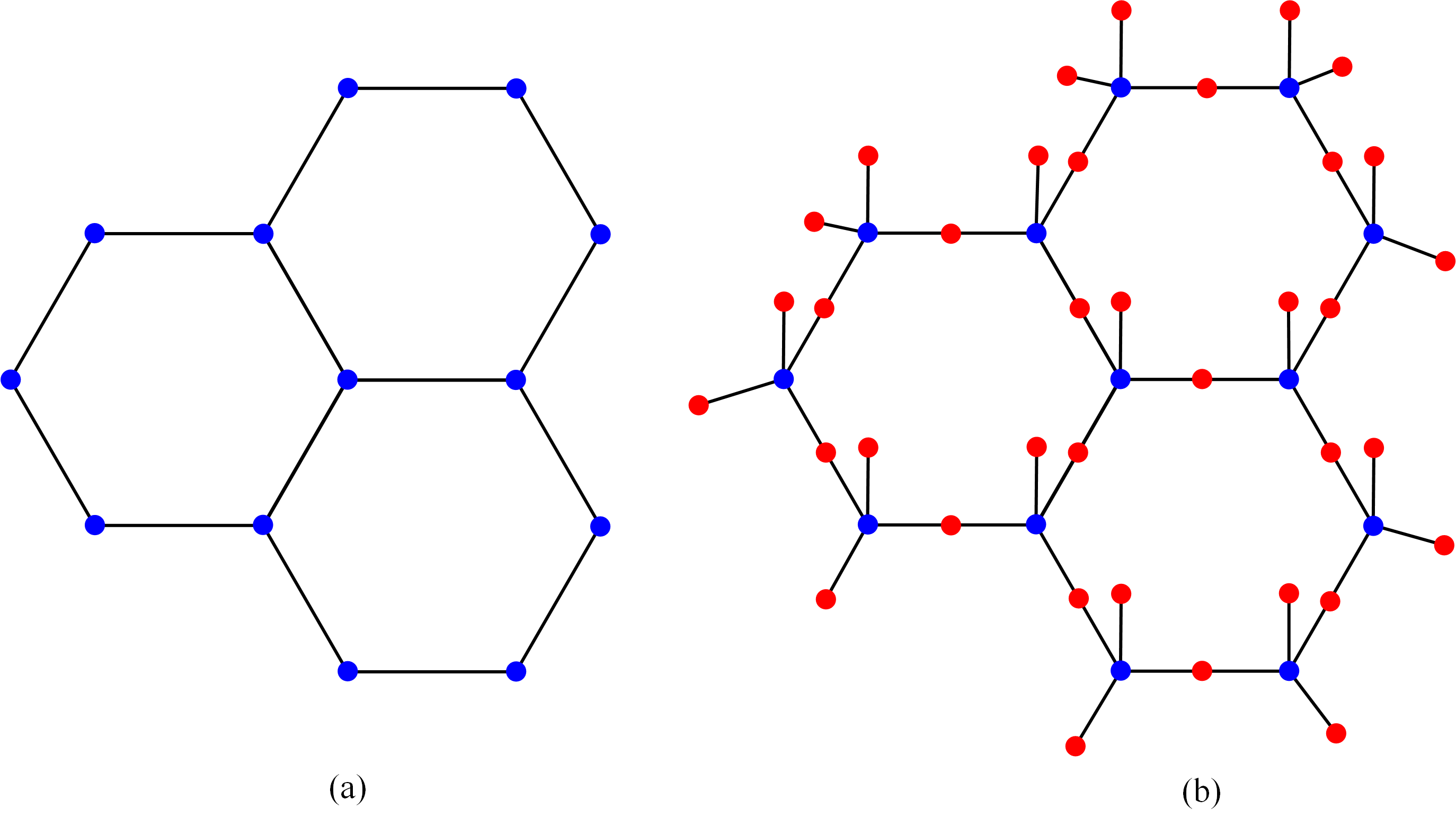}}
	\end{center}
	\caption{(a) Hexagonal carbon nanosheet (b) Silicate sheet derived from hexagonal carbon nanosheet. Red vertices are oxygen atoms and blue vertices are silicon atoms.}
	\label{fig:FHexSil1}
\end{figure}

 An $r$-layer hexagonal carbon nanosheet was defined and studied by Deng et al.~\cite{DeCh14}. Three examples of hexagonal carbon nanosheets are displayed in Fig.~\ref{fig:F123LayHexCarbon}. Formally, a hexagonal carbon nanosheet consists of the central hexagon and $r-1$ layers of hexagons around it. We mention that Deng et al.~\cite{DeCh14} call the hexagonal carbon nanosheet as {\em zigzag-graphene (6-Z-HGNS)}, while Xiaoxiao and Zhang~\cite{XiZh17} call it a {\em hexagonal model with alternate B- and N-terminated edges}.  
\begin{figure}[ht!]
	\begin{center}
		\scalebox{0.6}{\includegraphics{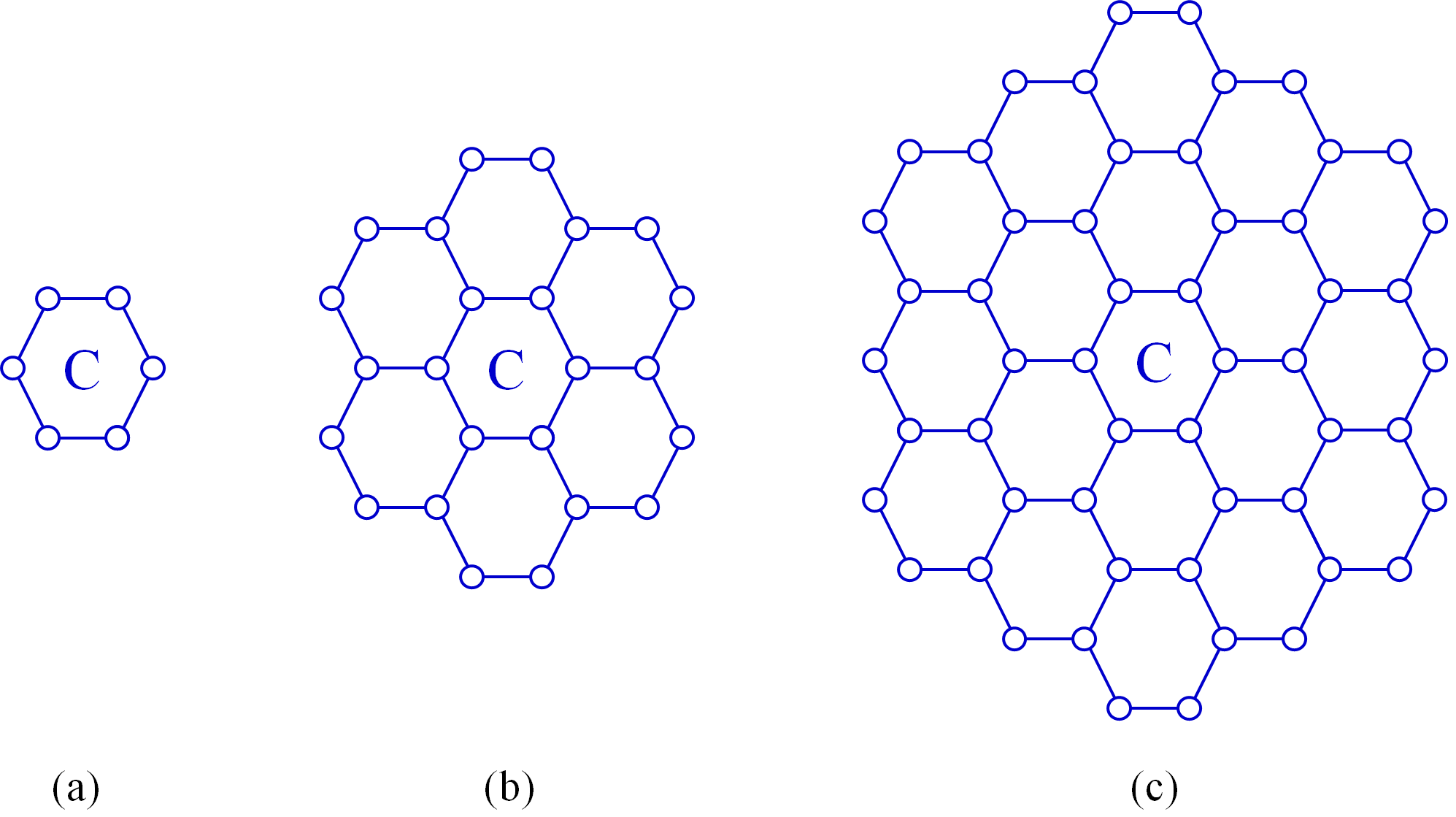}}
	\end{center}
	\caption{(a) 1-layer hexagonal carbon nanosheet, denoted by $C$. \quad (b) $2$-layer hexagonal carbon nanosheet. \quad (c) $3$-layer hexagonal carbon nanosheet.}
	\label{fig:F123LayHexCarbon}
\end{figure}

Next, we present the smoothing operation, which is used in the definition of the main concept of this section. Given a vertex $v$ of degree $2$ in $G$, let $x$ and $y$ be the neighbors of $v$ in $G$. {\em Smoothing a vertex} $v$  is the operation of removing the vertex $v$ and adding edge $xy$ to the graph $G$~\cite{MaBr21}.   

\begin{defn}
	Let $H$ denote a silicate sheet (Fig.~\ref{fig:FHexSil1}), and let $\SM(H)$ be the graph obtained from $H$ by smoothing all $2$-degree vertices of $H$. If the graph obtained from $\SM(H)$ by deleting all the pendant vertices of $\SM(H)$ is an $r$-layer hexagonal carbon nanosheet (see Fig.~\ref{fig:F123LayHexCarbon}), then $H$ is an {\em $r$-layer hexagonal silicate sheet}.
\end{defn}
\begin{figure}[ht!]
	\begin{center}
		\scalebox{0.6}{\includegraphics{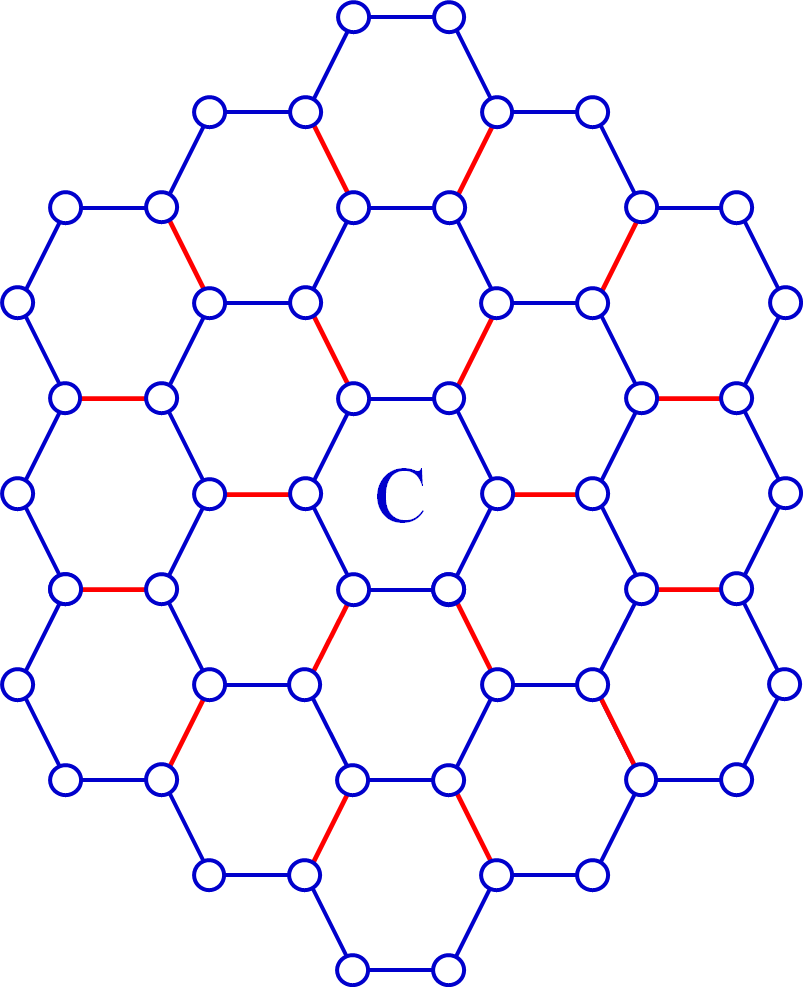}}
	\end{center}
	\caption{The $3$-layer hexagonal carbon nanosheet $G$ with $C$ as the central hexagon of $G$. The edges of the matching $M$ are red.}
	\label{fig:F1HatHatSil}
\end{figure}

\begin{thm}
	\label{LgtGgtSG}
	If $H$ is an $r$-layer hexagonal silicate sheet, then $\gpack(H) = \gt(H)$.
\end{thm}
\begin{proof}
	Let $G$ be the $r$-layer hexagonal carbon nanosheet, where $C$ is the central hexagon of $G$ (see Fig.~\ref{fig:F123LayHexCarbon}).
	Let us consider the following matching in $G$: $M$ = \{$xy\in E(G) : d(x, C) = 2k$ and $d(y, C)$ = $2k+1$, where $k\in\{0, 1,\dots, r-1\}$\} (see Fig.~\ref{fig:F1HatHatSil}). Given a graph $G$ and a matching $M$, the respective graph $G_M$ is defined in Section \ref{sec:geo-pack-leaf-attach}. Note that $G_M$ and $\SM(H)$ are isomorphic graphs. 
	In addition, the set $M'=\{x\in V(M):\, d(x,C)\textrm{ is odd}\}$ is independent, hence enjoying the conditions in Proposition~\ref{prp:gpackGM}.
	Thus, $\gpack(G_M)$ = $\gpack(\SM(H))$ = $\gt(G_M)$ = $\gt(\SM(H))$. Since $\gt(\SM(H))$ = $\gt(H)$ and $\gpack(\SM(H))$ = $\gpack(H)$, we conclude that $\gpack(H)$ = $\gt(H)$.
\end{proof}

Let $G$ be an $r$-layer hexagonal carbon nanosheet, and $M$ the matching as defined above. Then $|M|= 6(1 +  \ldots + r-1) = 3(r^2-r)$.  In addition, $|V(H)| = 6r^2$, cf. also~\rm \cite{ChSh90}. Now, it is straightforward to see: $$\gt(H) = |V(H)| - |M| = 6r^2 - 3(r^2-r) = 3(r^2+r).$$ 
\section{Conclusion}
In this paper, we solved the geodesic transversal problem for the well-known parallel architecture mesh of trees and some important chemical structures such as silicate networks and carbon nanosheets. We demonstrated some of the challenges of this combinatorial problem. In addition, see Fig.~\ref{FGrid_SemiDiagDrid_DiagGrid}, which presents three basic types of grids (of size $7\times 5$) in which the problem behaves very differently.  The geodesic transversal problem is trivial for the $2$-dimensional square grids, where the geodesic transversal number equals $2$ (see Proposition~\ref{Tgt2DGrid}). But the geodesic transversal problems for semi-diagonal grids and diagonal grids seem to be more difficult and remain open. We believe that $\gt(G) = r+s-1$ where $G$ is an $r\times s$-dimensional semi-diagonal or diagonal grid. 
\begin{figure}[ht!]
	\begin{center}
		\scalebox{0.5}{\includegraphics{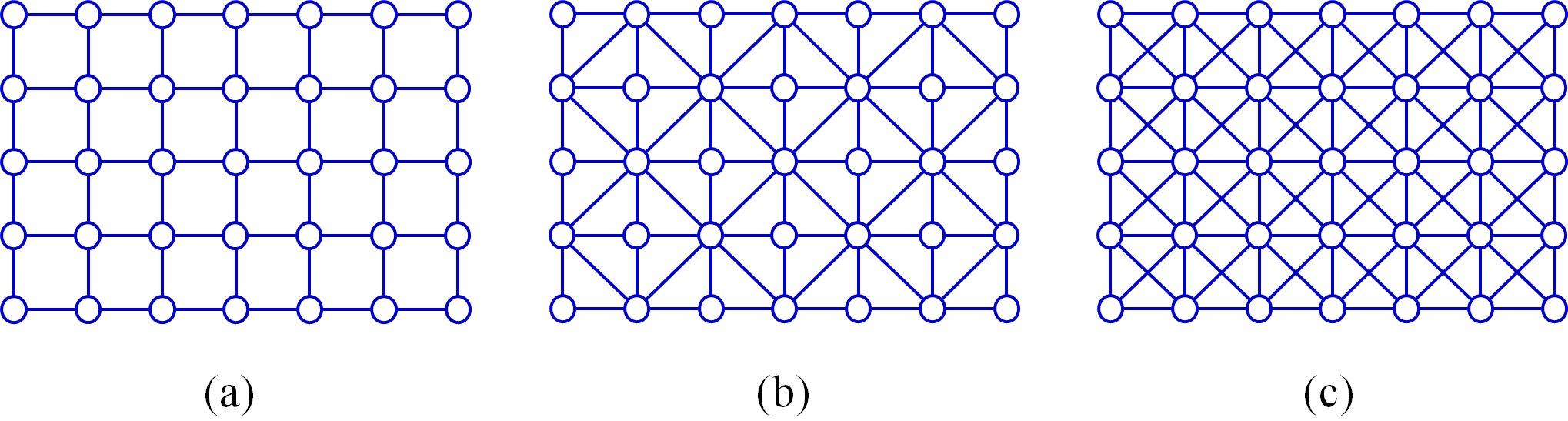}}
	\end{center}
	\caption{(a) Square grid. \quad (b) Semi-diagonal grid. \quad (c) Diagonal grid.}
	\label{FGrid_SemiDiagDrid_DiagGrid}
\end{figure}
While we believe these two cases are solvable, the geodesic transversal problem in more complex product graphs and many other relevant classes of graphs offer further challenging problems. 

In addition, the complexity status of this problem is unknown for intersection graphs such as chordal graphs, circular-arc graphs, permutation graphs etc and Cayley graphs such as wrapped butterfly,  circulant graphs, shuffle-exchange etc.
\section*{Acknowledgments}

This work was supported and funded by Kuwait University, Research Grant No.\ (QI 01/20).


\end{document}